\documentclass[10pt,a4paper]{amsart}

\usepackage{amsfonts}
\usepackage{amsthm}
\usepackage{amsmath}
\usepackage{amscd}
\usepackage[latin2]{inputenc}
\usepackage{t1enc}
\usepackage[mathscr]{eucal}
\usepackage{indentfirst}
\usepackage{graphicx}
\usepackage{graphics}
\usepackage{pict2e}
\numberwithin{equation}{section}
\usepackage[margin=2.9cm]{geometry}

\usepackage{pgfplots}
\usepackage{hyperref}

\newcommand\ts[1]{{\textstyle{#1}}}
 
\theoremstyle{plain}
\newtheorem{Th}{Theorem}[section]
\newtheorem{Lemma}[Th]{Lemma}

\newtheorem{Prop}[Th]{Proposition}

\theoremstyle{definition}

\newtheorem{Rem}[Th]{Remark}
\newtheorem{?}[Th]{Problem}

\newcommand{\Span}{\operatorname{span}}

\newcommand{\N}{\mathbb{N}}

\newcommand{\cD}{\mathcal{D}}
\newcommand{\cM}{\mathcal{M}}
\newcommand{\cO}{\mathcal{O}}

\begin{document}


\title{Decay of the Kolmogorov $N$-width for wave problems}
\author{Constantin Greif and Karsten Urban}
\address{Ulm University, Institute of Numerical Mathematics, Helmholtzstr.\ 20, D-89081 Ulm, Germany} 
\email{\{constantin.greif,karsten.urban\}@uni-ulm.de}

\begin{abstract} 
The Kolmogorov $N$-width $d_N(\mathcal{M})$ describes the rate of the worst-case error (w.r.t.\ a subset $\cM\subset H$ of a normed space $H$) arising from a projection onto the best-possible linear subspace of $H$ of dimension $N\in\N$. Thus, $d_N(\mathcal{M})$ sets a limit to any projection-based approximation such as determined by the reduced basis method. While it is known that $d_N(\mathcal{M})$ decays exponentially fast for many linear coercive parametrized partial differential equations, i.e., $d_N(\mathcal{M})=\cO(e^{-\beta N})$,  we show in this note, that only $d_N(\mathcal{M}) =\cO(N^{-1/2})$ for initial-boundary-value problems of the hyperbolic wave equation with discontinuous initial conditions. This is aligned with the known slow decay of $d_N(\cM)$ for the linear transport problem.
\end{abstract}


\keywords{Kolmogorov $N$-width, wave equation}
\subjclass[2010]{41A46, 65D15}
\maketitle

\section{Introduction}\label{Sec:1}
The Kolmogorov $N$-width is a classical concept of (nonlinear) approximation theory as it describes the error arising from a projection onto the best-possible space of a given dimension $N\in\N$, \cite{MR774404}. This error is measured for a class $\cM$ of objects in the sense that the \emph{worst} error over $\cM$ is considered. Here, we focus on subsets $\cM\subset H$, where $H$ is some Banach or Hilbert space with norm $\| \cdot \|_{H}$. Then, the Kolmogorov $N$-width is defined as
\begin{align} \label{defKolNwidth}
    d_N(\cM) := \inf\limits_{V_N\subset H;\ \dim V_N=N} 
    		\sup\limits_{ u \in \mathcal{M} } 
			\inf\limits_{v_N \in V_N }  \| u - v_N \|_{H},
\end{align}
where $V_N$ are linear subspaces. The corresponding approximation scheme is nonlinear as one is looking for the best possible linear space of dimension $N$. Due to the infimum, the decay of $d_N(\cM)$ as $N\to\infty$ sets a lower bound for the best possible approximation of all elements in $\cM$ by a linear approximation in $V_N$.

Particular interest arises if the set $\cM$ is chosen as a set of solutions of certain equations such as partial differential equations (PDEs), which is the reason why sometimes (even though slightly misleading) $\cM$ is termed as `solution manifold'. In that setting, one considers a \emph{parameterized} PDE (PPDE) with a suitable solution $u_\mu$ and $\mu$ ranges over some parameter set $\cD$, i.e., $\cM\equiv\cM(\cD):=\{ u_\mu:\, \mu\in\cD\}$, where we will skip the dependence on $\cD$ for notational convenience. As a consequence, the decay of the Kolmogorov $N$-width is of particular interest for model reduction in terms of the reduced basis method. There, given a PPDE and a parameter set $\cD$, one wishes to construct a possibly optimal linear subspace $V_N$ in an offline phase in order to highly efficiently compute a reduced approximation with $N$ degrees of freedom (in $V_N$) in an online phase. For more details on the reduced basis method, we refer the reader e.g.\ to the recent surveys \cite{Haasdonk:RB,RozzaRB,QuarteroniRB}.

It has been proven that for certain linear, coercive parameterized problems, the Kolmogorov $N$-width decays exponentially fast, i.e., 
\begin{align*}
    d_N( \mathcal{M} ) \leq C e^{-\beta N}
\end{align*}
with some constants $C<\infty$ and $\beta>0$, see e.g.\ \cite{MR2877366,OR16}. This extremely fast decay is at the heart of any model reduction strategy (based upon a projection to $V_N$) since it allows us to chose a very moderate $N$ to achieve small approximation errors. It is worth mentioning that this rate can in fact be achieved numerically by determining $V_N$ by a greedy-type algorithm.

However, the situation dramatically changes when leaving the elliptic and parabolic realm. In fact, it has been proven in \cite{OR16} that $d_N$ decays for certain first-order linear transport problems at most with the rate $N^{-1/2}$. This in turn implies that projection-based approximation schemes for transport problems severely lack efficiency, \cite{MR3911721,MR3177860}. In this note, we consider hyperbolic problems and show in a similar way as in \cite{OR16} that 
\begin{align*}
     d_N(\mathcal{M}) \geq \ts{\frac{1}{4}}\, N^{-1/2},
\end{align*}
(see Thm.\ \ref{maintheoremabschatzung} below) for an example of the second-order wave equation. In Section \ref{Sec:2}, we describe the Cauchy problem of a second-order wave equation with discontinuous initial conditions and review the distributional solution concept. Section \ref{Sec:3} is devoted to the investigation of a corresponding initial-boundary-value problem and Section \ref{Sec:4} contains the proof of Thm.\ \ref{maintheoremabschatzung}.

\section{Distributional solution of the wave equation on $\mathbb{R}$}\label{Sec:2}

We start by considing the univariate wave equation on the spatial domain $\Omega := \mathbb{R}$ and on the time interval $I := \mathbb{R}^+$ (i.e., a Cauchy problem) for a real-valued parameter $\mu \geq 0$ with discontinuous initial values, i.e.,
\begin{subequations} 
	\label{simple_waveequation}
	\begin{align} 
     	 \partial_{tt} u_{\mu}(t,x)  - \mu^2 \, \partial_{xx} u_{\mu}(t,x) 
      	&= 0 \quad \text{for} \quad (t,x) \in \Omega_I := I \times \Omega, \label{simple_wave_wobound}  \\
 	u_{\mu}(0,x) 
		&= u_0(x) :=  \begin{cases}
    			1, & \text{if $x<0$}, \\
    			-1, & \text{if $x\geq 0$},
 		 \end{cases}  \quad  x \in \Omega, 
		 \label{simple_wave_ini1} \\ 
  	\partial_{t} u_{\mu}(0,x) 
	&= 0, \quad  x \in \Omega. \label{simple_wave_ini2}
\end{align}
  \end{subequations}
This initial value problem has no classical solution, so that we consider a weak solution concept, namely we look for solutions in the distributional sense, which is known to be appropriate for hyperbolic problems.

\begin{Lemma} \label{lemma_dis_sol_wave}
A distributional solution of \eqref{simple_waveequation} is given, for $(t,x) \in \Omega_I = \mathbb{R}^+ \times \mathbb{R}$, by\\
\begin{minipage}{0.65\textwidth}
 \begin{align*}
    u_{\mu}(t,x) = \begin{cases}
    1, & \text{if $x < - \mu t $}, \\
     -1, & \text{if $x \geq \mu t $}, \\
     0, & \text{else}.
  \end{cases}
 \end{align*}
\end{minipage}
\begin{minipage}{0.3\textwidth}
  	 	\begin{tikzpicture}[x=10mm,y=10mm]
		\draw[->] (-1.3,0) -- (1.3,0) node[right] {\footnotesize$x$};
  		\draw[->] (0,0) -- (0,1.3) node[above] {\footnotesize$t$};
		\fill[opacity=0.5, black!20!white] (-1.2,0) -- (0,0) -- (-1.2,1.2) -- (-1.2,0);
		\fill[opacity=0.5, black!40!white] (0,0) -- (-1.2,1.2) -- (1.2,1.2) -- (0,0);
		\fill[opacity=0.5, black!30!white] (1.2,0) -- (0,0) -- (1.2,1.2) -- (1.2,0);
		\node at (-0.75,0.4) {$1$};
		\node at (0.8,0.4) {$-1$};
		\node at (0.2,0.7) {$0$};
		\draw (0,0) -- (-1.2,1.2) node [pos=0.9,above] {\tiny$t\!=\!-\frac{x}{\mu}$};
		\draw (0,0) -- (1.2,1.2) node [pos=0.9,above] {\tiny$t\!=\!\frac{x}{\mu}$};
		\end{tikzpicture}
\end{minipage}
\end{Lemma}
\begin{proof}
We start by considering the following initial value problem 
\begin{align} \label{funwaveeq} 
\begin{split}
      \partial_{tt} G_{\mu}(t,x)  - \mu^2 \cdot \partial_{xx} G_{\mu}(t,x) = 0 \quad  \text{for} \quad (t,x) \in \Omega_I ,  \\
 G_{\mu}(0,x) = 0, \quad \partial_{t} G_{\mu}(0,x) = \delta(x ),  \quad x \in \Omega,
\end{split}
\end{align}
where $\delta(\cdot)$ denotes Dirac's $\delta$-distribution at 0. A solution $G_{\mu}$ of \eqref{funwaveeq} is called \emph{fundamental solution} (see e.g. \cite[Ch.\ 5]{MR2028503}) and can easily be seen to read $G_{\mu}(t,x) = \frac{1}{2 \mu} \big(H(x+ \mu t) - H(x- \mu t)\big)$, 
where $H(x):= \int^{x }_{-\infty} \delta(y) dy $ denotes the Heaviside step function with distributional derivative $H' = \delta$. Hence, the distributional derivative of $G_{\mu}$ w.r.t.\ $t$ reads 
\begin{align} \label{Ablfundsol}
   \partial_t G_{\mu}(t,x) = \frac12 \big(\delta(x+\mu t) + \delta(x-\mu t)\big)
\end{align}
and it is obvious that $G_{\mu}(0,x) = 0$ as well as $\partial_{t} G_{\mu}(0,x) = \delta(x )$ for $x \in \mathbb{R}$. By using the properties of the Dirac's $\delta$-distribution (see e.g.\ \cite{MR1275833}) we observe that $\partial_{tt} G_{\mu}(t,x) = \frac{\mu}{2} \big( \delta(x+\mu t) - \delta(x-\mu t)\big)$ and $\partial_{xx} G_{\mu}(t,x) = \frac{1}{2 \mu} \big(\delta(x+\mu t) - \delta(x-\mu t) \big)$ in the distributional sense. Hence, $G_{\mu}$ satisfies \eqref{funwaveeq}. 

Now, we consider the original problem \eqref{simple_waveequation}. To this end, the following relation of the fundamental solution $G_{\mu}$ of \eqref{funwaveeq} and the solution $u_{\mu}$ of \eqref{simple_waveequation} is well-known \cite{MR2028503},
\begin{align*}
    u_{\mu} (t,x) = 
    	& \int _{\mathbb{R}} \partial_t G_{\mu}(t,x-y) u_{\mu}(0,y) dy 
    		+ \int _{\mathbb{R}} G_{\mu}(t,x-y) \partial_t u_{\mu}(0,y) dy. 
\end{align*}
Finally, inserting $\partial_t  G_{\mu}$ from \eqref{Ablfundsol}, the initial condition $u_{\mu}(0,\cdot) = u_0(\cdot)$ in $\mathbb{R}$, and the Neumann initial condition $ \partial_t  u_{\mu}(0,\cdot) = 0$ in $\mathbb{R}$, yields
\begin{align*}
    u_{\mu} (t,x) 
    	&= \ts{\frac12} \int _{\mathbb{R}} \big(  \delta(x-y+\mu t) + \delta(x-y-\mu t)\big ) u_0(y)\, dy    \\
   	&= \ts{\frac{1}{2}} \Big[  u_{0}(x+\mu t ) + u_{0}(x - \mu t ) \Big] = \begin{cases}
    		1, & \text{if $x < - \mu t $}, \\
     		-1, & \text{if $x \geq \mu t $}, \\
     		0, & \text{else},
  \end{cases} 
\end{align*}
which proves the claim.
\end{proof}

\section{The wave equation on the interval}\label{Sec:3}

Let us consider the wave equation \eqref{simple_wave_wobound}, but now on the bounded space-time domain $\Omega_I := (0,1) \times (-1,1)$ with Dirichlet boundary conditions
\begin{align} \tag{\ref{simple_waveequation}d} \label{21drandhi}
     u_{\mu}(t,-1) = 1,\quad 
     u_{\mu}(t,1) = -1,
  	  \qquad \text{for} \quad t \in I:= (0,1),
\end{align}
and the initial conditions (\ref{simple_wave_ini1},\ref{simple_wave_ini2}). It is readily seen that the  functions $ \varphi_{\mu}$ defined by\\
\begin{minipage}{0.65\textwidth}
\begin{align} \label{phiss}
    \varphi_{\mu}(t,x) := \begin{cases}
    1, & \text{if $x < - \mu t $}, \\
     -1, & \text{if $x \geq \mu t $}, \\
     0, & \text{else},
  \end{cases}
\end{align}
\end{minipage}
\begin{minipage}{0.3\textwidth}
  	 	\begin{tikzpicture}[x=10mm,y=10mm]
		\draw[->] (-1.1,0) -- (1.1,0) node[right] {\footnotesize$x$};
  		\draw[->] (0,0) -- (0,1.2) node[above] {\footnotesize$t$};
		\fill[opacity=0.5, black!20!white] (-1.0,0) -- (0,0) -- (-1.0,1.0) -- (-1.0,0);
		\fill[opacity=0.5, black!40!white] (0,0) -- (-1.0,1.0) -- (1.0,1.0) -- (0,0);
		\fill[opacity=0.5, black!30!white] (1.0,0) -- (0,0) -- (1.0,1.0) -- (1.0,0);
		\node at (-0.75,0.4) {$1$};
		\node at (0.75,0.4) {$-1$};
		\node at (0.2,0.7) {$0$};
		\draw (0,0) -- (-1.0,1.0) node [pos=0.9,above] {\tiny$t\!=\!-\frac{x}{\mu}$};
		\draw (0,0) -- (1.0,1.0) node [pos=0.9,above] {\tiny$t\!=\!\frac{x}{\mu}$};
		\draw[thick] (-1,0) -- (-1,1) -- (1,1) -- (1,0);
		\draw (-1,0.1) -- (-1,-0.1) node [below] {\tiny$-1$};
		\draw (1,0.1) -- (1,-0.1) node [below] {\tiny$1$};
		\draw (0.1,1) -- (-0.1,1) node [above] {\tiny$1$};
		\end{tikzpicture}
\end{minipage}
\newline
for $(t,x) \in \overline{\Omega}_I = [0,1] \times [-1,1]$ are contained in the solution manifold of (\ref{simple_waveequation}a-d), i.e., 
\begin{align} \label{solmanifoldM}
    \{ \varphi_{\mu}: \mu \in \mathcal{D} \} 
    \subset \mathcal{M} \equiv \mathcal{M}(\cD)
    := \{ u_{\mu}:  \mu \in \mathcal{D} := [0,1] \} \subset L_2 (\Omega_I). 
\end{align}
In fact, by Lemma \ref{lemma_dis_sol_wave}, $\varphi_{\mu}$ solves (\ref{simple_waveequation}a-c) on $\mathbb{R}^+ \times \mathbb{R}$  and they also satisfy the boundary conditions \eqref{21drandhi}.  
The next step is the consideration of a specific family of functions to be defined now. For some $M \in \mathbb{N}$ and $1 \leq m \leq M$, let 
 \begin{align} \label{psisortho}
    \psi_{M ,m}(t,x) := \begin{cases}
    1, & \text{if $x \in \big[- \frac{m }{M  } t, -\frac{m-1}{M } t \big) $}, \\
    -1, & \text{if $x \in \big[\frac{m-1}{M  } t, \frac{m }{M } t \big)  $},  \\
     0, & \text{else},
  \end{cases} 
  \quad \text{for} \quad (t,x) \in \bar{\Omega}_I ,
\end{align}
and we collect all $\psi_{M ,m}$, $m=1,\ldots,M$ in 
\begin{align} \label{Psifncset}
    \Psi_{M} := \{   \psi_{M,m}  :\,  1 \leq m \leq M \}.
\end{align}
Note, that $\Psi_{M}$ can be generated by
\begin{align} \label{Phifncset}
     \Phi_{M} :=\{  \varphi_{\frac{m}{M}}  :\,  0 \leq m \leq M  \} \subset \{ \varphi_{\mu}:\, \mu \in \mathcal{D} \} ,
\end{align}
as follows 
$\psi_{M ,m} = \varphi_{\frac{m-1}{M}} -  \varphi_{\frac{m}{M}}$, $1 \leq m \leq M$, 
which in fact can be easily seen; see also Figure~\ref{fig:test}. 
%
\begin{figure}[!htb]
  	 \begin{center}
	 	\begin{tikzpicture}
		\draw[->] (-1.2,0) -- (1.2,0) node[right] {\footnotesize$x$};
  		\draw[->] (0,-1.2) -- (0,1.2) node[above] {\footnotesize$y$};
		\draw[very thick] (-1,1) -- (0,1) -- (0,-1) -- (1,-1);
		\node at (0.7,1) {$\varphi_0$};
		\end{tikzpicture}
			\hspace*{5mm}
	 	\begin{tikzpicture}
		\draw[->] (-1.2,0) -- (1.2,0) node[right] {\footnotesize$x$};
  		\draw[->] (0,-1.2) -- (0,1.2) node[above] {\footnotesize$y$};
		\draw[very thick] (-1,1) -- (-0.16,1) -- (-0.16,0) -- (0.16,0) -- (0.16,-1) -- (1,-1);
		\node at (0.7,1) {$\varphi_{\frac13}$};
		\end{tikzpicture}
			\hspace*{5mm}
	 	\begin{tikzpicture}
		\draw[->] (-1.2,0) -- (1.2,0) node[right] {\footnotesize$x$};
  		\draw[->] (0,-1.2) -- (0,1.2) node[above] {\footnotesize$y$};
		\draw[very thick] (-1,1) -- (-0.33,1) -- (-0.33,0) -- (0.33,0) -- (0.33,-1) -- (1,-1);
		\node at (0.7,1) {$\varphi_{\frac23}$};
		\end{tikzpicture}
			\hspace*{5mm}
	 	\begin{tikzpicture}
		\draw[->] (-1.2,0) -- (1.2,0) node[right] {\footnotesize$x$};
  		\draw[->] (0,-1.2) -- (0,1.2) node[above] {\footnotesize$y$};
		\draw[very thick] (-1,1) -- (-0.5,1) -- (-0.5,0) -- (0.5,0) -- (0.5,-1) -- (1,-1);
		\node at (0.7,1) {$\varphi_1$};
		\end{tikzpicture}

	 	\begin{tikzpicture}
		\draw[->] (-1.2,0) -- (1.2,0) node[right] {\footnotesize$x$};
  		\draw[->] (0,-1.2) -- (0,1.2) node[above] {\footnotesize$y$};
		\draw[very thick] (-1,0) -- (-0.16,0) -- (-0.16,1) -- (0,1) -- (0,-1) -- (0.16,-1) -- (0.16,0) -- (1,0);
		\node at (0.7,1) {$\psi_{3,1}$};
		\end{tikzpicture}
			\hspace*{5mm}
	 	\begin{tikzpicture}
		\draw[->] (-1.2,0) -- (1.2,0) node[right] {\footnotesize$x$};
  		\draw[->] (0,-1.2) -- (0,1.2) node[above] {\footnotesize$y$};
		\draw[very thick] (-1,0) -- (-0.33,0) -- (-0.33,1) -- (-0.16,1) -- (-0.16,0)
				-- (0.16,0) -- (0.16,-1) -- (0.33,-1) -- (0.33,0) -- (1,0);
		\node at (0.7,1) {$\psi_{3,2}$};
		\end{tikzpicture}
			\hspace*{5mm}
	 	\begin{tikzpicture}
		\draw[->] (-1.2,0) -- (1.2,0) node[right] {\footnotesize$x$};
  		\draw[->] (0,-1.2) -- (0,1.2) node[above] {\footnotesize$y$};
		\draw[very thick] (-1,0) -- (-0.5,0) -- (-0.5,1) -- (-0.33,1) -- (-0.33,0)
				-- (0.33,0) -- (0.33,-1) -- (0.5,-1) -- (0.5,0) -- (1,0);
		\node at (0.7,1) {$\psi_{3,3}$};
		\end{tikzpicture}
  \caption{Top: functions $\varphi_\mu$ for $\mu=0,\frac13,\frac23,1$. Bottom: functions $\psi_{M,m}$ for $M=3$ and $m=1,2,3$. All for $t=\frac12$ fixed on $[-1,1]$.
  \label{fig:test}}
	\end{center}
\end{figure}
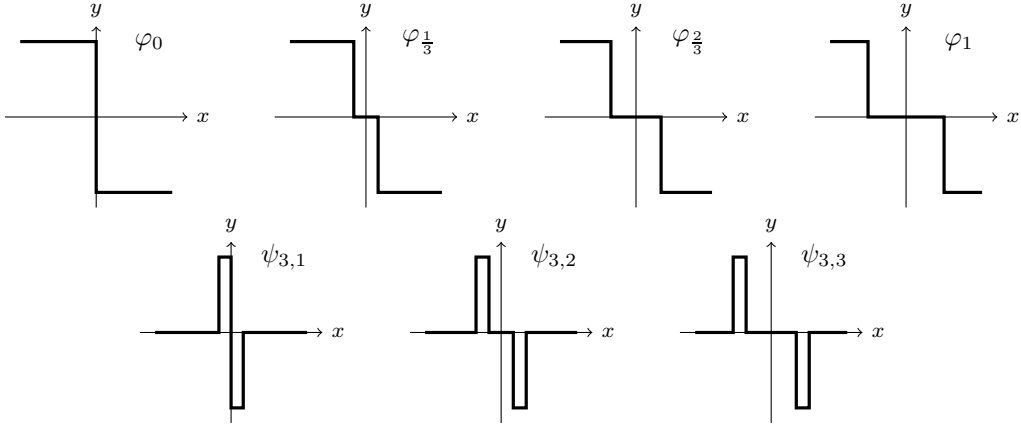
%
We will see later that $d_N(\Phi_M)\ge \frac12 d_N(\Psi_M)$. Moreover $\| \psi_{M,m} \|_{ L_2 (\Omega_I) } =  \sqrt{1/M} $ and these functions are pairwise orthogonal,  i.e.
\begin{align*}
 \big( \psi_{M,m_1}, \psi_{M,m_2} \big)_{L_2 (\Omega_I)} 
 	= \int\limits_{0}^1 \int\limits_{-1}^1 \psi_{M,m_1}(t,x) \ \psi_{M,m_2}(t,x) \ dx \ dt 
	= \ts{\frac{1}{M}}\,  \delta_{m_1, m_2},
\end{align*}
where $\delta_{m_1, m_2}$ denotes the Kronecker-$\delta$ for $m_1, m_2 \in \{ 1,\dots,M \}$. Thus,
\begin{align} \label{orthonormalf}
	\tilde{\Psi}_{M}:= \{   \tilde{\psi}_{M,m}  :\,  1 \leq m \leq M \}, 
	\qquad
  	\tilde{\psi}_{M,m} := \ts{\sqrt{M}}\, \psi_{M ,m}, 1 \leq m \leq M,
\end{align}
is a set of orthonormal functions.

\section{Kolmogorov $N$-width of sets of orthonormal elements}\label{Sec:4}
Let us start by introducing the notation $\mathcal{V}_N  := \{ V_N \subset H:\,  \text{linear space with } \text{dim}(V_N) = N  \}$, so that the Kolmogorov $N$-width in \eqref{defKolNwidth} can be rephrased as
\begin{align*} 
    d_N(\mathcal{M}) := \inf\limits_{V_N \in \mathcal{V}_N } \sup\limits_{ u \in \mathcal{M} } 
    	\inf\limits_{v_N \in V_N }  \| u - v_N \|_{H}.
\end{align*}
We are going to determine either the exact value or lower bounds of $d_N(\mathcal{M})$ for certain sets of functions.

\begin{Lemma} \label{lemmakolNeinheie}
The canonical orthonormal basis $\{ e_1, \dots, e_{2N} \}$ of $H := ( \mathbb{R}^{2N}, \| \cdot \|_2 )$ has the Kolmogorov $N$-width $d_N ( \{ e_1, \dots, e_{2N} \} )  =  \frac{1}{\sqrt{2}}$.
\end{Lemma}
\begin{proof}
Let $V_N = \{ v = \sum_{j=1}^N a_j d_j \ \vert \ a_1, \dots, a_N \in \mathbb{R} \} \in \mathcal{V}_N$, with $\{d_1, \dots, d_{N}\}$ being an arbitrary set of orthonormal vectors in $H$. Thus, $V_N$ is an arbitrary linear subspace of $H$ of dimension $N$. Then, for any $k \in \{1, \dots, 2N \} $ and the canonical basis vector $e_k \in \mathbb{R}^{2N}$, we get 
\begin{align*}
    \sigma_{V_N}(k)^2 
    	:= \inf_{v \in V_N} \| e_k - v \|_2^2 = \| e_k - P_{V_N}(e_k) \|_2^2  
	= \Big\| e_k - \sum_{j=1}^N ( d_j)_k d_j  \Big\|_2^2,
\end{align*}
where $P_{V_N}(e_k) = \sum_{j=1}^N \langle e_k,d_j \rangle d_j = \sum_{j=1}^N ( d_j)_k d_j $ is the orthogonal projection of $e_k$ onto $V_N$. Then, 
\begin{align*}
    \|  P_{V_N}(e_k) \|_2^2 
    	= \Big\langle \sum_{j=1}^N ( d_j)_k d_j, \sum_{l=1}^N ( d_l)_k d_l \Big\rangle 
	= \sum_{j=1}^N ( d_j)_k \Big\langle d_j, \sum_{l=1}^N ( d_l)_k d_l \Big\rangle   
	= \sum_{j=1}^N ( d_j)_k^2.
\end{align*}
Next, for $k \in \{1, \dots, 2N \} $ we get\footnote{We also refer to \cite{MR0109826,MR1971217}, where it was proven that $\|P \| = \| I-P \| $ for any idempotent operator $P\ne0$, i.e., \eqref{sumorthoproof}.}
\begin{align}\label{sumorthoproof}
    \sigma_{V_N}(k)^2 
    	&= \| e_k -  P_{V_N}(e_k) \|_2^2 
	= \|   P_{V_N}(e_k) \|_2^2  -  ( P_{V_N}(e_k))_k^2 + \big(1- ( P_{V_N}(e_k))_k \big)^2 
		\nonumber \\
    	&= \sum_{j=1}^N ( d_j)_k^2 - \Big(\sum_{j=1}^N ( d_j)_k^2 \Big)^2 
		+ 1 - 2 \sum_{j=1}^N ( d_j)_k^2 + \Big(\sum_{j=1}^N ( d_j)_k^2 \Big)^2 
	= 1 - \sum_{j=1}^N ( d_j)_k^2. 
\end{align}
Let us now assume that 
\begin{align} \label{contramustwrong}
    \sum_{j=1}^N ( d_j)_k^2 > \frac{1}{2} \quad \text{for all} \quad k \in \{1, \dots, 2N \}.
\end{align}
Then, we would have that
\begin{align*}
    N = \sum_{j=1}^N \| d_j \|_2^2 = \sum_{j=1}^N \sum_{k=1}^{2N}  ( d_j)_k^2  = \sum_{k=1}^{2N} \sum_{j=1}^N ( d_j)_k^2 > 2N \cdot \ts{\frac{1}{2}} = N,
\end{align*}
which is a contradiction, so that \eqref{contramustwrong} must be wrong and we conclude that there exists a $k^* \in \{1, \dots, 2N \}$ such that $ \sum_{j=1}^N ( d_j)_{k^*}^2 \leq \ts{\frac{1}{2}}$. This yields by \eqref{sumorthoproof} that $\sigma_{V_N}(k^*)^2 = 1 - \sum_{j=1}^N ( d_j)_{k^*}^2 \geq \ts{\frac{1}{2}}$. By using this $k^*$, this leads us to
\begin{align*}
  	d_N ( \{ e_1, \dots, e_{2N} \} ) 
  	=  \inf_{ V_N \in \mathcal{V}_N } \sup_{ k \in \{1, \dots, 2N \} } \inf_{v \in V_N} \| e_k - v \|_{2} 
	\geq  \inf_{ V_N \in \mathcal{V}_N} \sigma_{V_N}(k^*)  
	\geq \ts{\frac{1}{\sqrt{2}}}.
\end{align*}
To show equality, we consider $V_N:=\Span\{ d_j:\, j=1,\ldots,N\}$ generated by orthonormal vectors $d_j := \frac{1}{\sqrt{2}} ( e_{2j-1} + e_{2j})$. Then, for any even $k  \in \{2,4, \dots, 2N \}$ (and analogous for odd $k$) we get by \eqref{sumorthoproof} that
\begin{align*}
   \sigma_{V_N}(k)^2 = 1 - \sum_{j=1}^N ( d_j)_k^2 =  1 - \Big( \frac{1}{\sqrt{2}} ( e_{k-1} + e_{k}) \Big)_k^2 = 1 - \Big( \frac{1}{\sqrt{2}} \Big)^2 =  \frac{1}{2} ,
\end{align*}
which proves the claim.
\end{proof}

\begin{Rem}
We note that, more general, for $k \in \mathbb{N}$, it holds that $d_N ( \{ e_1, \dots, e_{k N} \} )  =  \ts{\sqrt{\frac{ k-1 }{ k}}}$, which can easily be proven following the above lines.
\end{Rem}

Having these preparations at hand, we can now estimate the Kolmogorov $N$-width for arbitrary orthonormal sets in Hilbert spaces. 
\begin{Lemma} \label{proprealtionsortho}
Let $H$ be an infinite-dimensional Hilbert space and $\{  \tilde{\psi}_1, \dots, \tilde{\psi}_{2 N}  \} \subset H$ any orthonormal set of size $2N$. Then, $d_N (\{  \tilde{\psi}_1, \dots, \tilde{\psi}_{2 N}  \})   = \frac{1}{\sqrt{2}}$. 
\end{Lemma}

\begin{proof}
Since $V_N := \arg\inf\limits_{V_N \in \mathcal{V}_N } \sup\limits_{ w \in \{  \tilde{\psi}_1, \dots, \tilde{\psi}_{2 N}  \} } \inf\limits_{v \in V_N }  \| w - v \|_{H} \subset \text{span} \{ \tilde{\psi}_1, \dots, \tilde{\psi}_{2 N} \}$,
we can consider the subspace $\Span \{ \tilde{\psi}_1, \dots, \tilde{\psi}_{2 N} \} \subset H$ instead of whole $H$.
The space $\text{span} \{ \tilde{\psi}_1, \dots, \tilde{\psi}_{2 N} \} $ with norm $ \| \cdot \|_{H} $ can be isometrically mapped to $\mathbb{R}^{2N} $ with canonical orthonormal basis $\{ e_1, \dots, e_{2N} \}$ and Euclidean norm $\| \cdot \|_2$. In fact, by defining the map $f : \text{span} \{ \tilde{\psi}_1, \dots, \tilde{\psi}_{2 N} \} \to \mathbb{R}^{2N}$  with $f(v) := \sum_{i = 1}^{2 N} ( v,  \tilde{\psi}_i )_{H}  \ e_i$. 
for $v,w \in \text{span} \{ \tilde{\psi}_1, \dots, \tilde{\psi}_{2 N} \}$ we get 
\begin{align*}
	\| f(w) - f(v) \|_2^2 
	&= \Big\| \sum_{i = 1}^{2 N} ( w-v,  \tilde{\psi}_i )_{H} \ e_i \Big\|_2^2 
	= \sum_{i = 1}^{2 N}  ( w-v,  \tilde{\psi}_i )_{H}^2 \|  e_i \|_2^2 
	= \sum_{i = 1}^{2 N}  ( w-v,  \tilde{\psi}_i )_{H}^2 \\
	&= \sum_{i = 1}^{2 N} ( w-v,  \tilde{\psi}_i )_{H}^2 \|  \tilde{\psi}_i \|_{H}^2 
	= \Big\| \sum_{i = 1}^{2 N} ( w-v,  \tilde{\psi}_i )_{H} \ \tilde{\psi}_i \Big\|_{H}^2 
	= \| w - v \|_{H}^2.
\end{align*}
Choosing $w = \tilde{\psi}_k, k \in \{ 1, \dots, {2 N} \}$, we have $f(w) = \sum_{i = 1}^{2 N} ( \tilde{\psi}_k ,  \tilde{\psi}_i )_{H}  e_i = e_k$. Thus, Lemma \ref{lemmakolNeinheie}, yields $d_N (\{  \tilde{\psi}_1, \dots, \tilde{\psi}_{2 N}  \})   =  d_N ( \{ e_1, \dots, e_{2N} \} ) = \frac{1}{\sqrt{2}}$, which proves the claim.
\end{proof}

\begin{Prop} \label{propo}
Let $\mathcal{M}$ be the solution manifold of (2.1a -- d) in \eqref{solmanifoldM} and $\Phi_{M}$, $\Psi_{M}$ defined in (\ref{Psifncset}, \ref{Phifncset}), $M\in\N$. Then, $d_N(\mathcal{M}) \geq  d_N (  \Phi_{M}  )  \geq \frac{1}{2}  d_N( \Psi_{M} )$ for $N \in \mathbb{N}$. 
\end{Prop}
\begin{proof}
By \eqref{solmanifoldM}, we have $\Phi_{M} =  \{  \varphi_{\frac{m}{M}}  :\,  0 \leq m \leq M  \} \subset \{ \varphi_{\mu} \ \vert  \ \mu \in \mathcal{D} \} \subset \mathcal{M}$, so that the first inequality is immediate. For the proof of the second inequality, we use the abbreviation $\| \cdot \| = \| \cdot \|_{L_2 (\Omega_I)}$. First, we denote some optimizing spaces and functions, $m \in \{ m^*-1, m^* \}$
\begin{align*}
	V_N^{\Psi_M} 
	&:= \arg\inf\limits_{V_N  \in \mathcal{V}_N} \sup \limits_{\psi \in \Psi_M } \inf \limits_{v  \in V_N }  \| \psi - v  \|,  
	&
 	\psi_{M,m^*} 
	&:= \arg \sup \limits_{\psi \in \Psi_M } \inf \limits_{v  \in V_N^{\psi} }  \| \psi - v  \|, \\
	V_N^{m} 
	&:= \arg\inf\limits_{V_N  \in \mathcal{V}_N} \inf \limits_{v  \in V_N }  \| \varphi_{\frac{m}{M}} - v  \|, 
	&
	v^{m }  
	&:= \arg \inf\limits_{v  \in V_N^{m}} \| \varphi_{\frac{m }{M}} - v  \|.
\end{align*}
With those notations, we get
\begin{align*}
    d_N ( \Psi_{M}) 
    &= \inf\limits_{V_N \in \mathcal{V}_N} \sup\limits_{ \psi \in \Psi_{M} } \inf\limits_{ v  \in V_N} \|  \psi - v  \| 
    =  \inf \limits_{v  \in V_N^{\Psi_M}  }  ||\psi_{M,m^*} - v  || \\
    &\kern-20pt\leq  \|  \psi_{M,m^*} - ( v^{m^*} -  v^{m^*-1} ) \| 
    =  \|  (\varphi_{\frac{m^*-1}{M}} - \varphi_{\frac{m^*}{M}}) - ( v^{m^*} -  v^{m^*-1} ) \|  
    \\
    &\kern-20pt\leq \| \varphi_{\frac{m^*-1}{M}} - v^{m^*-1}  \| +  \| \varphi_{\frac{m^* }{M}} - v^{m^* }  \|  
    =   \inf\limits_{v \in V_N^{m^*-1} } \| \varphi_{\frac{m^*-1}{M}} - v \| 
    		+  \inf\limits_{v \in V_N^{m^*} } \| \varphi_{\frac{m^* }{M}} - v \|   \\
    &\kern-20pt= \inf\limits_{V_N  \in \mathcal{V}_N} \inf\limits_{v \in V_N } \| \varphi_{\frac{m^*-1}{M}} - v \| 
    	+  \inf\limits_{V_N  \in \mathcal{V}_N} \inf\limits_{v \in V_N } \| \varphi_{\frac{m^* }{M}} - v \|  
     \leq \inf\limits_{v \in W_N} \| \varphi_{\frac{m^*-1}{M}} - v \| 
    	+  \inf\limits_{v \in W_N} \| \varphi_{\frac{m^* }{M}} - v \|,
\end{align*}
where $W_N :=  \arg\inf\limits_{V_N \in \mathcal{V}_N} \big( \inf\limits_{v \in V_N} \| \varphi_{\frac{m^*-1}{M}} - v \| + \inf\limits_{v \in V_N} \| \varphi_{\frac{m^* }{M}} - v \| \big) $. This gives
    \begin{align*}
    \inf\limits_{v \in W_N} \| \varphi_{\frac{m^*-1}{M}} - v \| +  \inf\limits_{v \in W_N} \| \varphi_{\frac{m^* }{M}} - v \| 
    & = \inf\limits_{V_N \in \mathcal{V}_N} \big( \inf\limits_{v \in V_N} \| \varphi_{\frac{m^*-1}{M}} - v \| + \inf\limits_{v \in V_N} \| \varphi_{\frac{m^* }{M}} - v \| \big)  \\
    & \leq   \inf\limits_{V_N \in \mathcal{V}_N} \big( 2 \sup\limits_{\varphi  \in  \Phi_{M} }  \inf\limits_{v \in V_N} \| \varphi  - v \| \big)  = 2 \cdot d_N ( \Phi_{M}),
    \end{align*}
which proves the second inequality.
\end{proof}

We can now prove the main result of this note.

\begin{Th} \label{maintheoremabschatzung}
For $\mathcal{M}$ being defined as in \eqref{solmanifoldM}, we have that $d_N(\mathcal{M}) \geq \frac{1}{4}\, N^{-1/2}$.
\end{Th}
\begin{proof}
Using Proposition \ref{propo} with $M = 2N$ (which in fact maximizes $d_N(\Psi_{M})$) yields $d_N(\mathcal{M}) \geq  d_N ( \Phi_{2N}  )  \geq \frac{1}{2} \cdot d_N(  \Psi_{2N} )$.
Since $V_N$ is a linear space, we have 
\begin{align*}
     d_N(  \Psi_{2N} ) 
     = d_N(\{  \psi_{2N,n} :\,  1 \leq n \leq 2N \}) 
     = \ts{\frac{1}{\sqrt{2N}}}\, d_N(\{  \sqrt{2N} \psi_{2N,n}  :\,  1 \leq n \leq 2N \}) 
     = \ts{\frac{1}{\sqrt{2N}}}\, d_N( \tilde{\Psi}_{2N} ).
\end{align*}
Applying now Lemma \ref{proprealtionsortho} for the orthonormal functions previously defined in \eqref{orthonormalf} gives 
$ \frac{1}{2}\, d_N(  \Psi_{2N} ) 
=  \frac{1}{2}\,  \frac{1}{\sqrt{2N}}  d_N(  \tilde{\Psi}_{2N} ) 
=   \frac{1}{2}  \frac{1}{\sqrt{2N}} \cdot \frac{1}{\sqrt{2}} 
= \frac{1}{{4} } N^{-1/2}$, 
which completes the proof.
\end{proof}

Theorem \ref{maintheoremabschatzung} shows the same decay of  $d_N (\mathcal{M} )$ as for linear advection problems, \cite{OR16}. Thus, transport and hyperbolic parametrized problems are expected to admit a significantly slower decay as for certain elliptic and parabolic problems as mentioned in the introduction. We note, that this result is \emph{not} limited to the specific discontinuous initial conditions \eqref{simple_wave_ini1}. In fact, also for continuous initial conditions with a smooth `jump', one can construct similar orthogonal functions like \eqref{psisortho} yielding the slow decay result.

\bibliographystyle{abbrv}
\bibliography{GU-WaveKolm_references.bib}

\end{document}